\documentclass[11pt]{amsart}
\usepackage[top=4.2cm, bottom=4cm, left=2.4cm, right=2.4cm]{geometry}
\usepackage[utf8]{inputenc}
\usepackage[USenglish]{babel}
\usepackage[T1]{fontenc} 
\usepackage{mathrsfs}
\usepackage{wasysym}
\usepackage{tikz-cd}
\usepackage{mathtools}
\usepackage{amsmath}
\usepackage{amssymb,epsfig}
\usepackage{amsthm}
\usepackage[absolute]{textpos}
\usepackage{mathtools,emptypage}

\mathtoolsset{showonlyrefs}  

\usepackage[bookmarks=true]{hyperref}
\usepackage{xcolor}
\hypersetup{
    colorlinks,
    linkcolor={red!50!black},
    citecolor={blue!50!black},
    urlcolor={blue!80!black},
}

\newcommand{\sfd}{{\sf d}}
\newcommand{\restr}[1]{\lower3pt\hbox{$|_{#1}$}}
\newcommand{\ca}[1]{\overset{\circ}{{#1}}}

\newcommand{\Kliminf}{K\kern-3pt-\kern-2pt\mathop{\rm lim\,inf}\limits}  
\newcommand{\Lip}{\mathop{\rm Lip}\nolimits}          
\renewcommand{\d}{{\mathrm d}}

 \newcommand{\X}{{\rm X}}
 
  \newcommand{\R}{{\mathbb R}}

\newcommand{\limi}{\varliminf}

\newcommand{\LIP}{{\rm LIP}}
\newcommand{\lip}{{\rm lip}}
\newcommand{\loc}{{\rm loc}}

\newcommand{\mm}{\mathfrak m}                                

\newtheorem{theorem}{Theorem}[section]

\newtheorem{corollary}[theorem]{Corollary}
\newtheorem{lemma}[theorem]{Lemma}
\newtheorem{proposition}[theorem]{Proposition}

\theoremstyle{definition}
\newtheorem{remark}[theorem]{Remark}
\newtheorem{definition}[theorem]{Definition}

\newtheorem{example}[theorem]{Example}
\newtheorem{question}[theorem]{Question}

\begin{document}

\title[Approximation by BV-extension sets via perimeter minimization]{Approximation by BV-extension sets\\ via perimeter minimization in metric spaces}

\author{Jesse Koivu}
\author{Danka Lu\v{c}i\'c}
\author{Tapio Rajala}

\address{University of Jyvaskyla \\
         Department of Mathematics and Statistics \\
         P.O. Box 35 (MaD) \\
         FI-40014 University of Jyvaskyla \\
         Finland}

\email{jesse.j.j.koivu@jyu.fi}
\email{danka.d.lucic@jyu.fi}
\email{tapio.m.rajala@jyu.fi}

\date{\today}



\begin{abstract}
We show that every bounded domain in a metric measure space can be approximated in measure from inside by closed $BV$-extension sets. The extension sets are obtained by minimizing the sum of the perimeter and the measure of the difference between the domain and the set. By earlier results, in PI-spaces the minimizers have open representatives with locally quasiminimal surface. We give an example in a PI-space showing that the open representative of the minimizer need not be a $BV$-extension domain nor locally John. 
\end{abstract}

\maketitle

\section{Introduction}

In this paper we study the existence of $BV$-extension sets in complete and separable metric measure spaces $\X$. By $BV$-extension sets we mean sets $E$ for which any integrable function with finite total variation on $E$ can be extended to the whole space $\X$ without increasing the $BV$-norm by more than a constant factor. $BV$- and Sobolev-extension sets are useful in analysis because via the extension one can use tools a priori available only for globally defined functions also for the functions defined only in the extension set.
Not every domain of a space is an extension set, so in cases where one starts with functions defined on an arbitrary domain $\Omega$ one first approximates $\Omega$ from inside by an extension set, then restricts the functions to this set and then extends them as global functions. Such process immediately raises the question: when can we approximate a domain from inside by extension domains (or sets)?

In the Euclidean setting, an answer to this has been known for a long time.
For instance, from the works of Calder\'on and Stein \cite{cal1961,stein} we know that Lipschitz domains of $\mathbb R^n$ are $W^{1,p}$-extension domains for every $p\geq 1$. 
Any bounded domain in $\mathbb R^n$ can be easily approximated from inside and outside by Lipschitz domains. It was later observed that in a more abstract setting of PI-spaces (that is doubling metric measure spaces satisfying a local Poincar\'e inequality \cite{HK1998}; see Section \ref{sec:PI}), good replacements of Lipschitz domains are uniform domains.
In \cite{BS2007} it was shown that uniform domains in $p$-PI-spaces are $N^{1,p}$-extension domains, for $1 \le p < \infty$, for the Newtonian Sobolev spaces, and in \cite{L2015} it was shown that bounded uniform domains in 1-PI-spaces are BV-extension domains. Finally, in \cite{R2021} it was shown that in doubling quasiconvex metric spaces one can approximate domains from inside and outside by uniform domains. Since PI-spaces are quasiconvex \cite{C1999,K2003}, we conclude that in PI-spaces one can approximate domains by extension domains.

Recently there has been increasing interest in analysis in metric measure spaces $(\X,\sfd,\mm)$ without the PI-assumption.
However, the extendability of $BV$-functions seems to have been studied only in some specific cases, such as infinite dimensional Gaussian case \cite{BPS2014}. 
We continue into the direction of general metric measure spaces and show in Theorem \ref{thm:extensionsub} that even without the PI-assumption one can still approximate domains $\Omega$ from inside by closed $BV$-extension sets. It is not clear if an approach similar to the approximation by uniform domains could work in general metric measure spaces. Therefore, we take a completely different approach and obtain the extension set by minimizing the functional $A \mapsto {\rm Per}(A) + \lambda\mm(\Omega \setminus A)$ for a large parameter $\lambda>0$.  Section \ref{sec:approximation} contains the proof of Theorem \ref{thm:extensionsub} and remarks on the minimization procedure. Before it, in Section \ref{sec:preli} we recall and prove preliminary results on $BV$-functions and sets of finite perimeter. In Section \ref{sec:PI} we connect the minimization approach to domains with locally quasiminimal boundary in PI-spaces, and also show that in PI-spaces the open representatives of the minimizers of the functional, and consequently domains with locally quasiminimal boundary need not be $BV$-extension domains, nor locally John domains. In the final part of the paper, Section \ref{sec:q} we list open questions raised by our extension result.

\section{Preliminaries}\label{sec:preli}
We will always assume $(\X,\sfd,\mm)$ to be a metric measure space where $(\X,\sfd)$ is a complete and separable metric space and $\mm$ is a Borel measure that is finite on bounded sets.
The set of all Borel subsets of \(\X\) is denoted by \(\mathscr B(\X)\).
We define the open and the closed ball with center \(x\in \X\) and radius \(r>0\) by 
\[
B_r(x)\coloneqq \{y\in \X:\, {\sf d}(x,y)<r\}\quad
\text{ and }\quad
\bar{B}_r(x)\coloneqq \{y\in \X:\, {\sf d}(x,y)\leq r\},
\]
respectively. We shall denote by \({\rm LIP}(\X)\) the space of all Lipschitz functions on \(\X\) and by \({\rm Lip}(f)\) the (global) Lipschitz constant of \(f\in {\rm LIP}(\X)\). 
Given any \(f\in {\rm LIP}(\X)\) and  \(E\subset \X\) we set 
\({\rm Lip}(f;E)\coloneqq {\rm Lip}(f\restr{E})\). Having this notation at our disposal, the \emph{asymptotic Lipschitz constant} (or the \emph{asymptotic slope}) of a function \(f\in 
{\rm LIP}(\X)\) is a function  \({\rm lip}_a(f)\colon \X\to [0,+\infty)\) given by 
\[\lip_a(f)(x)\coloneqq \inf_{r>0}{\rm Lip}\big(f;B_r(x)\big)\quad \text{ for every }x\in \X.\]
Notice also that \({\rm lip}_a(f)\leq {\rm Lip}(f)\).
Given an open set \(A\subset \X\)
we will say that a function \(f\colon \X\to \R\)
is \emph{locally Lipschitz} on \(A\) if for every \(x\in A\) there exists \(r>0\) such that \(B_r(x)\subseteq A\) and 
\(f\restr{B_r(x)}\) is Lipschitz.
We denote the space of all locally Lipschitz functions on \(A\) by \({\rm LIP}_{\loc}(A)\).
\medskip

\paragraph{\bf Functions of bounded variation.}
We next recall the definition of  the space of functions of bounded variation (BV functions, for short), as well as some of the characterisations of the total variation (measure) associated with a BV function. The below presentation is based on \cite{DMthesis}.

\begin{definition}[Total variation]
Let $(\X,\sfd,\mm)$ be a metric measure space. Consider $f \in L^1_{\rm loc}(\mm)$. Given an open set $A \subset \X$, we define
\begin{equation*}
    |D f|_{\X}(A):= \inf \left\{ \liminf_{n\to\infty} \int_A \lip_a (f_n)\, \d \mm:\, f_n \in {\rm LIP}_{\rm loc}(A),\, f_n \to f \in L^1_{\rm loc}(\mm \restr{A}) \right\}. 
\end{equation*}
\end{definition}
We extend $|D f|_\X$ to all Borel sets as follows: given $B \in \mathscr{B}(\X)$, we define
\[ |D f|_\X:= \inf \left\{ |D f|_\X(A), B \subset A, A \text{ is an open set}\right\}. \]
With this construction, $|D f|_\X \colon \mathscr{B}(\X) \to [0,\infty)$ is a Borel measure, called the \emph{total variation measure} of $f$ (\cite[Thm.\ 3.4]{Mir03}).
It follows from the definition that, given an open set $A \subset \X$
\begin{equation}
    \label{eq:lsc_totvar_open}
    f_n \to f \quad \text{in }L^1_{\textrm{loc}}(\mm\restr A) \qquad \Rightarrow \qquad |D f|_\X(A) \le \limi_{n \to \infty} |D f_n|_\X(A).
\end{equation}
%
Given a Borel set $B\subset \X$ and $f \in L^1_{\textrm{loc}}(\mm\restr{B})$, we introduce the following notation:
\[
|D f|_B\coloneqq \text{ the total variation measure of \(f\) computed in the metric measure space }(\X,\sfd,\mm\restr{B}).
\]

\begin{definition}[The spaces $\ca{BV}(B)$ and $BV(B)$]
Let $(\X,\sfd,\mm)$ be a metric measure space. Let $B \subset \X$ be Borel. 
We define
\[
\begin{split}
\ca{BV}(B):=&\big\{ f\in L^1_{\loc}(\mm\restr{B}):\,|D f|_B(B)<+\infty \big\},\\
BV(B):=&\big\{ f \in L^1(\mm\restr{B}):\,|D f|_B(B)<+\infty\big\}.
\end{split}
\]
We endow the space \(\ca{BV}(B)\) 
with the seminorm
 and the space 
\(BV(B)\) with the norm given by 
\[\|f\|_{\ca{BV}(B)}:= |D f|_B(B)\quad\text{and}\quad
\|f\|_{BV(B)}:=\|f\|_{L^1(\mm\restr{B})}+|D f|_B(B),
\]
respectively.
\end{definition}

\begin{remark}\label{rmk:lipX}
 The following characterisation of the total variation measure of the whole space will be useful for our purposes. 
 By \cite[Theorem 4.5.3]{DMthesis} we have that
\begin{equation}\label{eq:lip}
    |D f|_{\X}(\X):= \inf \left\{ \liminf_{n\to \infty} \int_\X \lip_a (f_n)\, \d \mm:\, f_n \in {\LIP}(\X),\, f_n \to f \in L^1_{\loc}(\mm) \right\}.
\end{equation}
 In general, we cannot restrict to globally Lipschitz functions when calculating the total variation measure: consider $A = (0,1)\cup(1,2) \subset \mathbb R$ and $f = \chi_{(0,1)}$. 
\end{remark}

We will use the following version of Lipschitz extensions where the asymptotic Lipschitz constant is preserved.

\begin{proposition}[{\cite[Theorem 1.1]{DMGP2020}}]\label{prop:globallipextension}
Let $(\X,\sfd)$ be a metric space, $C \subset \X$ a subset and $g \colon C \to \mathbb R$ a Lipschitz function. Then for every
$\varepsilon > 0$ there exists an $(\Lip(g) + \varepsilon)$-Lipschitz function $f \colon \X \to \mathbb R$ whose restriction to $C$ coincides with $g$ and such that
\[
\lip_a (g)(x) = \lip_a(f)(x) \qquad \text{ for every }x \in C.
\]
Moreover if $g$ is bounded (resp. with bounded support), then $f$ can be chosen to be bounded (resp. with bounded
support).
\end{proposition}

By combining Proposition \ref{prop:globallipextension} with Remark \ref{rmk:lipX} we get the following. 

\begin{corollary}\label{cor:closed}
Let \((\X,\sfd,\mm)\) be a metric measure space. Let $B \subset \X$ be closed and define $Y = (B,\sfd|_{B\times B},\mm|_B)$. Then $BV(B) = BV(Y)$ and the total variation measures $|Df|_B$ and $|Df|_Y$ agree on the Borel subsets of $B$ for every $f \in BV(B)$. Moreover,
\begin{equation}\label{eq:lip2}
    |D f|_B(B) = \inf \left\{ \liminf_{n\to \infty} \int_B \lip_a (f_n)\, \d \mm:\, f_n \in {\rm LIP}(\X),\, f_n \to f \in L^1_{\rm loc}(\mm \restr{B}) \right\}. 
\end{equation}
\end{corollary}
\begin{proof}
By Proposition \ref{prop:globallipextension} every $f \in {\LIP}(B)$ can be extended to an element of ${\LIP}(\X)$ without changing the asymptotic Lipschitz constant on $B$, thus (taking into account Remark \ref{rmk:lipX}) we obtain 
\begin{equation}\label{eq:fromextension}
|Df|_B(\X) =|Df|_Y(Y),
\end{equation}
and thus \(BV(B)=BV(Y)\) (cf.\ \cite[Theorem 3.1]{DMGP2020}).

Now, take $A \subset X$ open. Since every
$f \in {\LIP}_{\rm loc}(A)$ can be restricted to an element of ${\LIP}_{\rm loc}(B\cap A)$, we get that
\begin{equation}\label{eq:trivial}
|Df|_Y(B \cap A) \le |Df|_B(A).
\end{equation}
By the definition of total variation measure, the inequality \eqref{eq:trivial} extends to all Borel sets $A \subset \X$.
Finally, by \eqref{eq:fromextension} and recalling that $|Df|_Z$ is a finite Borel measure for any metric measure space \((Z,\sfd_Z,\mm_Z)\), we have for all Borel $A \subset \X$ that
\[
|Df|_B(\X) \le |Df|_Y(B) = |Df|_Y(A\cap B) + |Df|_Y(B\setminus A) \le |Df|_B(A) + |Df|_B(\X\setminus A) = |Df|_B(\X)
\]
giving the equality
\[
|Df|_Y(A\cap B) = |Df|_B(A).
\]
The equality \eqref{eq:lip2} follows by taking \(A=B\)
in the above equality,
combined with Remark \ref{rmk:lipX} and Proposition \ref{prop:globallipextension}.
\end{proof}

We define the notion of sets of finite perimeter on a Borel subset $B \subset \X$.
\begin{definition}[Sets of finite perimeter on a Borel subset B]
    Let $(\X,\sfd,\mm)$ be a metric measure space and $B,E\in\mathscr B(\X)$. 
    We define the \emph{perimeter of} \(E\) \emph{on} \(B\) as 
    \[{\rm Per}_B(E)\coloneqq |D\chi_E|_B(B).\]
    We say that $E$ has finite perimeter on $B$ if the quantity ${\rm Per}_{B}(E)$ is finite.
    Moreover, we define for every \(F\in \mathscr B(\X)\) the quantity  ${\rm Per}_B(E;F):= |D \chi_E|_B(B \cap F)$.
\end{definition}
To shorten the notation, whenever \(B\) is equal to  the whole (base) space \(\X\), we will often write 
    \({\rm Per}(E)\) instead of \({\rm Per}_\X(E)\).
\medskip

\paragraph{\bf Extension sets and extension properties.}

\begin{definition}[\(BV\)-extension set]
A set \(B\in \mathscr B(\X)\) is said to be a \(BV\)-extension set if there exist
\(C>0\) and a map \(E_B\colon BV(B)\to BV(\X)\), 
such that for every \(f\in BV(B)\) the following hold:
\begin{itemize}
    \item [i)] \(\|E_Bf\|_{BV(\X)}\leq C\|f\|_{BV(B)}\);
    \smallskip
    
    \item [ii)] \(E_Bf\restr{B}=f\).
\end{itemize}
\end{definition}

Given a \(BV\)-extension set \(B\), we define the operator norm of  \(E_B\) as
\[
\|E_B\|\coloneqq 
\inf \big\{c\geq 0:\, \|E_Bf\|_{BV(\X)}\leq c\,\|f\|_{BV(B)}\, \text{ holds for all } f\in BV(B)\big\}.
\]

\begin{definition}[Extension property for sets of finite perimeter]
Let \(B\in \mathscr B(\X)\). We say that
\(B\) has \emph{the extension property for sets of finite perimeter} with respect to the full \(BV\)-norm if there exists \(C>0\) such that  
for every \(E\subset B\) with \({\rm Per}_B(E)<+\infty\) there exists \(\widetilde E\in \mathscr B(\X)\) such that the following two properties hold:
\begin{itemize}
    \item [i)] \(\mm(\widetilde E)+{\rm Per}(\widetilde E)\leq C\big(\mm(E)+{\rm Per}_B(E)\big)\)
    \smallskip
    
    \item [ii)] \(\mm(E\Delta (\widetilde E\cap B))\)=0.
\end{itemize}
    
\end{definition}

\section{Approximation by BV-extension sets from inside}
\label{sec:approximation}

In this section we prove the main result of the paper, Theorem \ref{thm:extensionsub}, according to which we can estimate domains from inside by closed $BV$-extension sets. In the proof we will need the following two results. The first one connects the extendability of $BV$-functions with the extendability of sets of finite perimeter. In the Euclidean case, such result was obtained by Burago and Maz’ya \cite{BM1969}. Later it was extended to PI-spaces by Baldi and Montefalcone \cite{BM2008}. The connection of perimeter- and $BV$-extensions with $W^{1,1}$-extensions was studied in detail in \cite{GBR2022} in Euclidean spaces, and then in general metric measure spaces in \cite{CKR2023}. In \cite[Proposition 3.4]{CKR2023} the extension result closest to what we need was proven. There a Borel set was shown to be a $BV$-extension set if and only if it has the extension property for Borel sets of finite perimeter with the full norm. We need to make a small modification to this result, since in our proof we need to stay in the class of closed sets and consequently will only use open sets for testing the perimeter extensions.

\begin{proposition}\label{prop:BV}
Let \((\X,\sfd,\mm)\) be a metric measure space. A Borel subset $\Omega \subset \X$ has the extension property for $BV$ if and only if it has the extension property for open sets of finite perimeter with the full norm.
\end{proposition}
\begin{proof}
Having already the equivalence between $BV$-extension and perimeter extension of Borel sets given by \cite[Proposition 3.4]{CKR2023}, we only need to show that perimeter extension for open sets implies $BV$-extension for functions in $BV(\Omega) \cap L^\infty(\Omega)$.
Towards this, take $f \in BV(\Omega) \cap  L^\infty(\Omega)$. By the definition of the total variation, there exists a sequence of open sets $U_n \supset \Omega$ and functions $f_n \in {\rm LIP}_{\rm loc}(U_n)$ such that $f_n \to f$ in $L^1_{\rm loc}(\mm \restr{\Omega})$ and
\[
\liminf_{n\to \infty} \int_\Omega \lip_a (f_n)\, \d \mm = |Df|_\Omega(\Omega).
\]
Now, by assumption we can extend each relatively open set $A_{n,t} = \{x \in \Omega \,:\, f_n(x) > t\}$ to a Borel set $\tilde A_{n,t} \subset \X$ so that
\[
\mm(\tilde A_{n,t}) + {\rm Per}_\X(\tilde A_{n,t}) \le C\left( \mm(A_{n,t}) + {\rm Per}_\Omega(A_{n,t})\right),
\]
where $C>0$ is the constant given by the assumption on having the extension property for open sets.

As in the proof of \cite[Proposition 3.4]{CKR2023}, this implies that we get an extension $\tilde f_n \in BV(\X)$ of $f_n$ with 
\[
\|\tilde f_n\|_{BV(\X)} \le C \|f_n\|_{BV(\Omega)}.
\]
By an application of Mazur's lemma (see again the proof of \cite[Proposition 3.4]{CKR2023} for details), this implies that we also get an extension $\tilde f \in BV(\X)$ of $f$ with 
\[
\|\tilde f\|_{BV(\X)} \le C \|f\|_{BV(\Omega)}.
\]
This concludes the proof.
\end{proof}

The next lemma is the reason why our approach works only for closed sets. Later in Example \ref{ex:simple} we observe that the claim of the lemma fails for general sets $B\subset \X$.

\begin{lemma}\label{lma:perimetersum}
Let \((\X,\sfd,\mm)\) be a metric measure space. Given a closed set \(B \subset \X\) and a set \(A \subset B\) of finite perimeter on \(B\), it holds that 
\begin{equation}\label{eq:perimetersum}
{\rm Per}(A)+{\rm Per}(B\setminus A)
\leq {\rm Per}(B)+2{\rm Per}_B(A).
\end{equation}
\end{lemma}
\begin{proof}
Let \((f_i)_i\subseteq {\rm LIP}(\X)\) be such that 
\begin{equation}\label{eq:aprox_B}
f_i\to \chi_B\,\text{ in }L^1_{\rm loc}(\mm)\quad \text{ and }\quad \lim_{i\to\infty}\int_\X{\rm lip}_a(f_i)\,{\rm d}\mm ={\rm Per}(B).
\end{equation}
Since $B$ is closed, by Corollary \ref{cor:closed} there exists a sequence \((g_i)_i\subseteq {\LIP}(\X)\)
such that 
\begin{equation}\label{eq:aprox_A}
g_i\to \chi_A\,\text{ in }L^1_{\rm loc}(\mm\restr{B})
\quad \text{ and }\quad \liminf_{i\to\infty}\int_{B}{\rm lip}_a(g_i)\,{\rm d}\mm = {\rm Per}_B(A).
\end{equation}
For a fixed \(i\in \mathbb N\) we then have that
\[
\lim_{j\to +\infty} \int_{\X \setminus B} f_j\,{\rm lip}_a(g_i)\, {\rm d}\mm =0.
\]
Therefore, 
up to taking a (relabeled) 
subsequence of  \((f_i)_i\),  we may assume that
\begin{equation}\label{eq:workingtogether}
\lim_{i\to +\infty} \int_{\X \setminus B} f_i\,{\rm lip}_a(g_i)\, {\rm d}\mm =0.
\end{equation}
Now notice that 
$f_i\,g_i \to \chi_A$ and $f_i(1-g_i) \to \chi_{B \setminus A}$ in $L_{\text{loc}}^1(\mm)$, and that
\begin{align*}
 &\int_\X \text{lip}_a(f_i\,g_i)\,{\rm d}\mm 
 + \int_\X \text{lip}_a(f_i\,(1-g_i))\,{\rm d}\mm\\
  \le & \int_\X (f_i\,\text{lip}_a(g_i) + g_i\,\text{lip}(f_i))\,{\rm d}\mm + \int_\X (f_i\,\text{lip}_a(1-g_i) + (1-g_i)\,\text{lip}_a(f_i))\,{\rm d}\mm\\
  = & \,2 \int_\X f_i\, \text{lip}_a(g_i)\,{\rm d}\mm 
  + \int_\X \text{lip}_a(f_i)\,{\rm d}\mm.
\end{align*}
Taking into account \eqref{eq:workingtogether}, this gives
\begin{align*}
\text{Per}(A) + \text{Per}(B\setminus A) & \le \liminf_{i\to \infty} 
2 \int_\X f_i \text{lip}_a(g_i)\,{\rm d}\mm + \int_\X \text{lip}_a(f_i)\,{\rm d}\mm\\
&= \liminf_{i\to \infty} 
2 \int_B f_i \text{lip}_a(g_i)\,{\rm d}\mm + \int_\X \text{lip}_a(f_i)\,{\rm d}\mm
\leq  2 \text{Per}_B(A) + \text{Per}(B),
\end{align*}
where the last inequality follows from \eqref{eq:aprox_A} and \eqref{eq:aprox_B}.
\end{proof}

Notice that Lemma \ref{lma:perimetersum} does not hold in general if we replace the closed set \(B\) with a general Borel set. This is seen from the next simple example.
\begin{example}\label{ex:simple}
Let us consider \((\R,\sfd_{\rm Eucl},\mathcal L^1)\) as our metric measure space. Let 
\(B = (0,1)\cup (1,2)\) and \(A=(0,1)\). Then we have that
\[
4={\rm Per}(A)+{\rm Per}(B\setminus A)>{\rm Per}(B)+ 2{\rm Per}_B(A)=2.
\]
\end{example}

\begin{theorem}\label{thm:extensionsub}
 Let $(\X,\sfd,\mm)$ be a metric measure space. Let $\Omega \subset \X$ be a bounded open set. Then for every $\varepsilon >0$ there exists a closed set $G \subset \Omega$ such that $\mm(\Omega\setminus G) < \varepsilon$ and so that the zero extension gives a bounded operator from $BV(G)$ to $BV(\X)$.
\end{theorem}

\begin{proof}
 Let us denote  \(\mathcal C_\Omega = \{A \subset \Omega \,:\, A\text{ closed}\}\). 
 We consider the following functionals. For $\lambda > 0$ define \(M_\lambda \colon \mathcal C_\Omega\to [0,+\infty]\) as
 \[
     M_\lambda(A)\coloneqq \text{Per}(A) + \lambda \mm(\Omega \setminus A).
 \]
 We will show that for $\lambda$ large enough, a minimal element in a partial order given by $M_\lambda$ will give the desired set $G$. We divide the proof into several steps.
\smallskip

 \noindent {\color{blue}\textsc{Step 1}}: For every $\lambda>0$, we have $\inf_{A \in \mathcal C_\Omega}M_\lambda(A) < +\infty$. Moreover, given \(\varepsilon>0\)
 there exists \(\lambda>0\) such that  
 for any sequence $(A_i^\lambda)_i \subset \mathcal C_\Omega$ satisfying $M_\lambda(A_i^\lambda) \to \inf_{A \in 
 \mathcal C_\Omega}M_\lambda(A)$ as $i \to +\infty$ we have that
\begin{equation}\label{eq:measure_condition}
\lim_{i\to +\infty}
\frac{1}{\lambda}
M_\lambda(A_i^\lambda)
<\varepsilon.
\end{equation}
\begin{proof}[Proof of Step 1]
For every \(r>0\) we set 
\[B(\partial \Omega, r)\coloneqq 
\{x\in X:\,{\rm dist}(\partial \Omega, x)<r\}\quad 
\text{ and }\quad m_r\coloneqq \mm(\Omega\cap B(\partial \Omega, r)).\]
Consider the truncated distance function 
\({\rm dist}_r(\cdot,\partial \Omega)
\coloneqq {\rm dist}(\cdot, \partial \Omega)\wedge r\).
By Coarea formula we have that 
\[
|D{\rm dist}_r(\cdot, \partial \Omega)|(\Omega)=\int_0^r{\rm Per}(\{{\rm dist}_r(\cdot, \partial \Omega)>s\};\Omega)\,{\rm d}s=\int_0^r{\rm Per}(\Omega\setminus B(\partial \Omega, s))\,{\rm d}s.
\]
Moreover, \(|D{\rm dist}_r(\cdot, \partial \Omega)|(\Omega)=|D{\rm dist}_r(\cdot, \partial \Omega)|(\Omega\cap B(\partial \Omega, r))\leq \mm(\Omega\cap B(\partial \Omega, r))=m_r\). Together with the above, this 
gives the existence of $s \in [0,r]$ such that 
\[{\rm Per}(\Omega \setminus B(\partial\Omega,s)) \le \frac{m_r}{r},\]
proving the first part of the claim.
Let now $\varepsilon > 0$. Take $r>0$ so small that $m_r = \mm(B(\partial\Omega,r)\cap \Omega)<\frac{\varepsilon}{2}$. 
Note that for any $\lambda > 0$, we have
\(
\lim_{i\to \infty} M_\lambda(A_i^\lambda) \le \frac{m_r}{r} + \lambda m_r
\)
and so, by taking $\lambda > \frac1r$, we get
\[
 \lim_{i\to \infty} \frac{1}{\lambda}\,M_\lambda(A_i^\lambda) \le \frac{m_r}{\lambda r} + m_r < 2m_r < \varepsilon.
\]
This proves the claim of \textsc{Step 1}. 
\end{proof}

Next, we shall consider the following (non-empty, due to \textsc{Step 1}) subset of \(\mathcal C_\Omega\): 
\[\mathcal C_{\Omega,\lambda}\coloneqq  \{A \in \mathcal C_\Omega\,:\, M_\lambda(A) < +\infty\}.\]
Consider now a partial order \(A \prec_\lambda B\) on \(\mathcal C_{\Omega,\lambda}\) defined as
\[
A \prec_\lambda B\quad \text{ if and only if }\quad  \mm(A\setminus B) = 0\text{ and }M_\lambda(A) \le M_\lambda(B).
\]
\smallskip

\noindent
{\color{blue}\textsc{Step 2}}: For every $\lambda>0$
and \(C \in \mathcal C_{\Omega,\lambda}\),
the set \(\{A\in \mathcal C_{\Omega,\lambda}:\, A\prec_\lambda C\}\)
has a minimal element with respect to the partial order \(\prec_\lambda\).
\begin{proof}[Proof of Step 2]
 By Zorn's Lemma, it suffices to prove that any chain 
 \(
 (A_i^\lambda)_{i\in I}\subset 
\{A\in \mathcal C_{\Omega, \lambda}:\, A\prec_\lambda C\}
 \)
 contains a lower bound. 
 By selecting inductively elements in the chain so that $\mm(A_i^\lambda \setminus A_j^\lambda)>0$, we may assume that $I = \mathbb N$. Moreover, we may assume that $A_{i+1}^\lambda \subset A_i^\lambda$ for all $i \in \mathbb N$.
 We claim that
 \[
 A^\lambda = \bigcap_{i =1}^\infty A_i^\lambda
 \]
 gives the lower bound. Trivially, $A^\lambda \subset A_i^\lambda$ for all $i\in \mathbb N$, so it is enough to prove that 
 \(
  M_\lambda(A^\lambda) \le M_\lambda(A_i^\lambda)
 \)
 for all $i \in \mathbb N$.
 To verify the latter, notice that by the continuity of measure, we have that $\mm(A^\lambda) = \lim_{i\to +\infty}\mm(A_i^\lambda)$. 
 Consequently, $\chi_{A_i^\lambda} \to \chi_{A^\lambda}$ in $L^1(\X)$ and so by the lower semicontinuity of the perimeter, we have also
 \(
  \text{Per}(A^\lambda) \le \liminf_{i \to +\infty}\text{Per}(A_i^\lambda),
 \)
 proving the claim.
\end{proof}

We now show that for any $\lambda >0$ and a minimal element $G_\lambda \in \mathcal C_{\Omega,\lambda}$ with respect to $\prec_\lambda$ we have that the zero extension from $G_\lambda$ gives a bounded operator.
Given any Borel set \(B\subset \X\), in what follows we will denote  by \(E_B\)  the zero-extension operator from 
\( BV(B)\) to \(BV(\X)\).
\smallskip

\noindent
{\color{blue}\textsc{Step 3}}: 
Fix any \(\lambda>0\) and \(C\in \mathcal C_{\Omega, \lambda}\).
Let \(G_{\lambda,C}\) be a minimal element in \(\{A\in \mathcal C_{\Omega,\lambda}:\, A\prec_{\lambda} C\}\) 
with respect to the partial order 
\(\prec_\lambda\). Then we have that
$\|E_{G_\lambda}\| < +\infty$.
\begin{proof}[Proof of Step 3]
 By Proposition \ref{prop:BV}, we only need to check that the zero extension is bounded for characteristic functions of open sets of finite perimeter in $G_{\lambda,C}$.
So, let $A \subset G_{\lambda,C}$ be relatively open with $\text{Per}_{G_{\lambda,C}}(A)<+\infty$.
 Then by the minimality of $G_{\lambda,C}$ and the fact that \(\mm((G_{\lambda,C}\setminus A)\setminus G_{\lambda,C})=0\) 
 we have that
 \begin{equation}\label{eq:minimalitygives}
\text{Per}(G_{\lambda,C}) + \lambda \mm(\Omega \setminus G_{\lambda,C}) \le \text{Per}(G_{\lambda,C} \setminus A) + \lambda \mm(\Omega \setminus (G_{\lambda,C} \setminus A)),
 \end{equation}
 and by Lemma \ref{lma:perimetersum}
 \begin{equation}\label{eq:lemmagives}
  \text{Per}(G_{\lambda,C} \setminus A) + \text{Per}(A) \le
 \text{Per}(G_{\lambda,C}) + 2\text{Per}_{G_{\lambda,C}}(A).
 \end{equation}
 Therefore, combining \eqref{eq:minimalitygives} and \eqref{eq:lemmagives} we get
 \[
  \text{Per}(A) \le 2\text{Per}_{G_{\lambda,C}}(A)
 + \lambda \mm(A),
 \]
 and so $\|E_{G_{\lambda,C}}\| \le \max\{2,\lambda+1\}$ for characteristic functions.
\end{proof}
\smallskip

We are now ready to combine the results obtained in the three steps above and get the claim of the theorem. 

\noindent {\color{blue}\textsc{Step 4.}}
Fix \(\varepsilon>0\). There exists a closed set \(G\subset \Omega\) such that 
\[\mm(\Omega\setminus G)<\varepsilon\quad \text{ and }\quad \|E_G\|<+\infty.\]
\begin{proof}[Proof of Step 4.]
Let \(\lambda\) (depending on \(\varepsilon\)) be such that the claim of \textsc{Step 1} holds and fix any minimizing sequence \((A^\lambda_i)_{i\in \mathbb N}\). 
Then, for \(i\in \mathbb N\) large enough we have that \(\frac{1}{\lambda} M_\lambda(A_i^\lambda)<\varepsilon\) and thus \(A_i^\lambda\in \mathcal C_{\lambda,\Omega}\).
Let \(G_{\lambda,A_i^\lambda}\) be a minimal element 
in the set \(\{A\in \mathcal C_{\Omega, \lambda}:\, A\prec_\lambda A_i^\lambda\}\)
with respect to the partial order \(\prec_\lambda\), whose existence has been proved in \textsc{Step 2}. By \textsc{Step 3} we know that \(G_{\lambda,A_i^\lambda}\) is a \(BV\)-extension set, thus it only remains to check that \(\mm(\Omega\setminus G_{\lambda,A_i^\lambda})<\varepsilon\).
To verify this, notice that, by the minimality property of \(G_{\lambda,A^\lambda_i}\), it holds that
\[
\mm(\Omega\setminus G_{\lambda,A^\lambda_i})\leq \frac{1}{\lambda} M_\lambda(G_{\lambda,A^\lambda_i})\leq \frac{1}{\lambda} M_\lambda(A_i^\lambda)<\varepsilon.
\]
This proves the statement of \textsc{Step 4} (and of the theorem itself) for \(G=G_{\lambda,A^\lambda_i}\).
\end{proof}
\end{proof}

By approximating a measurable set from outside by an open set, Theorem \ref{thm:extensionsub} gives the following corollary.

\begin{corollary}\label{cor:measurable}
 Let $(\X,\sfd,\mm)$ be a  metric measure space and let $F \subset \X$ be a bounded Borel set. Then for every $\varepsilon >0$ there exists a closed set $G \subset \X$ such that $\mm(F\Delta G) < \varepsilon$ and so that the zero extension gives a bounded operator from $BV(G)$ to $BV(\X)$.
\end{corollary}

\begin{remark}\label{rmk:representatives}
A stronger version of Corollary \ref{cor:measurable} where we require in addition that $G \subset F$, does not hold. A counter example is given by taking $F$ to be a fat Cantor set in $\mathbb R$ equipped with the Lebesgue measure.
\end{remark}

We end this section with an example where the set $G$ does not have an open representative.

\begin{example}\label{ex:positiveboundary}
Let $X = \R^2$ with the Euclidean distance. We define $\Omega = Q \cup \bigcup_{n = 1}^{\infty} T_n$, where $Q = (0,1)\times (-1,0)$ and $T_n$ are defined as follows. We start by defining a triangle with unit length base:
\[
T = \big\{(x,y) \in \mathbb R^2\,:\, y \ge 0,\, y < x < 1 - y \big\}.
\]
Notice that $T$ contains the base, but not the other two sides of the triangle.
We then define
\[
T_n = \big( 2^{-2n+1}T + (2^{-2n+1},0)
\big)\quad \text{ and }\quad 
S = \bigcup_{n = 1}^{\infty} T_n.
\]
Let 
\[
w(x,y) = \begin{cases} 
      \min\big\{1, {\rm dist}\big((x,y), \R \times \{-1, 0\}\big)\big\}, & (x,y)\ {\rm with}\ y\in [-1,0] \\  
      1, & {\rm otherwise}.
\end{cases}
\]
Furthermore, define 
\[
\mm = w\mathcal{L}^2 + \sum_{n = 1}^{\infty} 2^{-n}\delta_{x_n},
\]
where $x_n = (2^{-2n+1} + 2^{-2n},0)$ is the center point of the base of the triangle $T_n$.

{\color{blue} \sc Step 1:} Let us show that we can split the functional $M_{\lambda}$ with respect to the cube $\overline{Q}$ and the triangles $T_n$.
First notice that for all  $A \subset \Omega$ we have
\[
\mm(\Omega \setminus A) = \mm (Q \setminus A) + \sum_{n=1}^\infty\mm({T}_n \setminus A).
\]
Towards showing that the perimeter part of the functional $M_\lambda$ also splits, 
we next show that for a finite perimeter set $A \subset \Omega$ it holds Per$(A\cap \R_{+}^{2}; \R_{+}^{2}) = {\rm Per}(A; \R_{+}^{2})$, where $\R_{+}^{2}=\R \times [0,\infty)$ is the closed upper half plane.
We do this by showing the chain of inequalities
\begin{equation}\label{Perimeterestimate}
\begin{split}
{\rm Per}(A) & = {\rm Per}(A;\R_{+}^{2}) + {\rm Per}(A;\R^2 \setminus \R_{+}^{2}) \\
& \ge {\rm Per}(A\cap \R_{+}^{2}; \R_{+}^{2}) + {\rm Per}(A;\R^2 \setminus \R_{+}^{2}) \ge {\rm Per}(A).
\end{split}
\end{equation}
The equality in the chain \eqref{Perimeterestimate} follows by subadditivity.
We first show the inequality Per$(A \cap \R_{+}^{2}; \R_{+}^{2}) \le {\rm Per}(A; \R_{+}^{2})$.
To this end we define 
\[
\phi_i(x) = {\rm min}\big\{0, 1 - 2i\,{\rm dist}(x,\{0\}\times \R)\big\}
\]
and call 
$U_i$ the $\frac{1}{i}$-neighborhood of $\R_{+}^{2}$.
This way we obtain $\phi_i \in  {\LIP}(\R^2)$ with values in $[0,1]$ such that spt$(\phi_i) \subset U_i$, $\phi_i \rightharpoonup \chi_{\R_{+}^{2}}$ 
in \(L^1(\mm)\)
and 
\[
\int_{\R^2} {\rm lip}_a (\phi_i)\,\d \mm \to 0.
\]
Further let  $f_i \in {\LIP}_{{\rm loc}}(U_i)$
be such that $f_i \rightharpoonup \chi_{A}$  and $\int_{\R^2} {\rm lip}_a (f_i)\,\d\mm \to {\rm Per}(A;\R_{+}^{2})$. We may assume that $f_i$ have values in $[0,1]$. Now setting $g_i = f_i \phi_i$ we have $g_i \rightharpoonup \chi_{A\cap \R_{+}^{2}}$ and $g_i$ is an admissible sequence of Lipschitz functions for Per$(A\cap \R_{+}^{2};\R_{+}^{2})$.
By the Leibniz rule we now obtain

\begin{equation}\label{leibnizestimate}
\begin{split}
\int_{\R^2} {\rm lip}_a (g_i)\,\d\mm &\le \int_{\R^2} \lvert \phi_i \rvert{\rm lip}_a (f_i)\,\d\mm + \int_{\R^2} \lvert f_i \rvert {\rm lip}_a (\phi_i)\,\d\mm \\
&\le \int_{\R^2} {\rm lip}_a (f_i)\,\d\mm + \int_{\R^2} {\rm lip}_a (\phi_i)\,\d\mm \to {\rm Per}(A;\R_{+}^{2}).
\end{split}
\end{equation}
Thus we have Per$(A \cap \R_{+}^{2}; \R_{+}^{2}) \le {\rm Per}(A; \R_{+}^{2})$. 

Next we show the second inequality ${\rm Per}(A \cap \R_{+}^{2}; \R_{+}^{2}) + {\rm Per}(A;\R^2 \setminus \R_{+}^{2}) \ge {\rm Per}(A)$.
To this end we let $\phi_i$ be as before. Further let $f_i \rightharpoonup \chi_{A}$ be a sequence of ${\LIP}_{{\rm loc}}(\R^2 \setminus \R_{+}^2)$ functions, such that $\int {\rm lip}_a (f_i)\,\d\mm \to {\rm Per}(A; \R^2 \setminus \R_{+}^{2})$, and let $g_i \rightharpoonup \chi_{A \cap \R_{+}^{2}}$ be a sequence of ${\LIP}_{{\rm loc}} (U_{i})$ functions, such that $\int {\rm lip}_a (g_i)\,\d\mm \to {\rm Per}(A\cap \R_{+}^{2};\R_{+}^{2})$.
We may again assume that $f_i$ and $g_i$ have values in $[0,1]$.
Therefore, we can set $h_i = \phi_i g_i + (1-\phi_i)f_i$, for which it holds $h_i \rightharpoonup \chi_{A}$.
Now again by a similar approximation as before using the Leibniz rule we obtain
\begin{equation}
\begin{split}
\int_{\R^2} {\rm lip}_a (h_i)\,\d\mm & \le \int\lvert \phi_i \rvert {\rm lip}_a (g_i)\,\d\mm + \int \lvert g_i \rvert {\rm lip}_a (\phi_i)\,\d\mm  + \int \lvert 1 - \phi_i \rvert {\rm lip}_a (f_i)\,\d\mm + \int \lvert f_i \rvert {\rm lip}_a (1-\phi_i)\,\d\mm \\
& \le \int {\rm lip}_a (g_i)\,\d\mm + 2\int {\rm lip}_a (\phi_i)\,\d\mm + \int {\rm lip}_a (f_i)\,\d\mm\\
&
\to {\rm Per}(A \cap \R_{+}^{2}; \R_{+}^{2}) + {\rm Per}(A;\R^2 \setminus \R_{+}^{2}),
\end{split}
\end{equation}
from which the claimed inequality follows.
Now Per$(A\cap \R_{+}^{2}; \R_{+}^{2}) = {\rm Per}(A; \R_{+}^{2})$.
Notice that since $\R^2 \setminus \R_{+}^{2}$ is an open set, we have
\[
{\rm Per}(A;\R^2 \setminus \R_{+}^{2}) = {\rm Per}(A \setminus \R_{+}^{2}; \R^2 \setminus \R_{+}^{2}).
\]
Let us recall that the perimeter measure enjoys the following locality property: given an open set \(U\subset \X\) and sets of finite perimeter \(E,F\subset \X\) such that \(\mm(U\cap (E\Delta F))=0\), it holds that 
\begin{equation}\label{eq:locality}
{\rm Per}(E;U)={\rm Per}(F;U). 
\end{equation}
Taking 
into account that $\overline{T}_n$ are pairwise disjoint compact sets together with \eqref{eq:locality}, one can easily verify that  
\[
 {\rm Per}(A\cap \R_{+}^{2};\R_{+}^{2})
=
\sum_{n=1}^\infty {\rm Per}(A\cap \overline{T}_n;\overline{T}_n).
\]
Consequently, we
get
\begin{equation}
\begin{split}
{\rm Per}(A) & = {\rm Per}(A;\R^2 \setminus \R_{+}^{2})+  {\rm Per}(A;\R_{+}^{2}) \\
&= {\rm Per}(A \setminus \R_{+}^{2};\R^2 \setminus \R_{+}^{2}) + {\rm Per}(A\cap \R_{+}^{2};\R_{+}^{2})\\
&= {\rm Per}(A \setminus \R_{+}^{2};\R^2 \setminus \R_{+}^{2}) + \sum_{n=1}^\infty {\rm Per}(A\cap \overline{T}_n;\overline{T}_n).
\end{split}
\end{equation}


{\color{blue} \sc Step 2:} Let $G_{\lambda}$ be a minimizer of $M_\lambda$.
We look to show that for large $\lambda > 0$ and $n > 0$, $G_{\lambda}$ will contain one of the points $x_n$, but nothing of the respective triangle $T_n$, in the measure sense, i.e. $\mm(G_{\lambda} \cap {T}_n) = 0$.
This means that $G_{\lambda}$ does not have an open representative.

Although it is not strictly needed in the following, we first notice that $\overline{Q} \subset G_{\lambda}$ as long as $\lambda > 0$ is large enough.
Next we will perform a reflection of the part of $G_{\lambda}$ that lies inside the triangles $T_n$ across the line $[0,1] \times \{0\}$.  Let $\widetilde{G}_{\lambda,n} = (G_{\lambda} \cap T_n) \cup \{(x,y)\subset \R^2: (x,-y)\in G_{\lambda} \cap T_n\}$. Now we estimate
\begin{equation}\label{reflect-isoper}
\begin{split}
{\rm Per}(G_{\lambda}\cap \overline{T}_n, \overline{T}_n) &\ge \frac{1}{2}{\rm Per}_{\rm euc}(\widetilde{G}_{\lambda,n};\R^2)\\
& \ge C\mathcal{L}^2(\widetilde{G}_{\lambda,n})^{\frac{1}{2}} \\
& = C'\mathcal{L}^2(G_{\lambda} \cap T_n)^{\frac{1}{2}},
\end{split}
\end{equation}
where ${\rm Per}_{\rm euc}$ denotes the Euclidean perimeter.
The first inequality follows since an admissible Lipschitz function for the definition of the perimeter on the left hand side will define an admissible Lipschitz function for the definition of the Euclidean perimeter on the right hand side via a reflection. For the second inequality 
we used the Euclidean isoperimetric inequality.
Now given $\lambda > 0$ and as long as $n > 0$ is large enough that $\mathcal{L}^2(G_{\lambda} \cap T_n)^{\frac{1}{2}} \le \frac{C'}{\lambda}$, it holds ${\rm Per}(G_{\lambda}\cap \overline{T}_n; \overline{T}_n) \ge \lambda\mathcal{L}^2(G_{\lambda} \cap T_n)$. Therefore, since $G_\lambda$ is a minimizer of $M_\lambda$, by Step 1  the set $G_\lambda \cap \overline{T}_n$ is a minimizer inside $\overline{T}_n$. Thus we
conclude that $\mathcal{L}^2(G_{\lambda} \cap T_n) = 0$. 
Since
${\rm Per}(\{x_n\},\overline{T}_n) = 0$
and $\mm(\{x_n\}) = 2^{-n}$, we have
$G_\lambda \cap \overline{T}_n= \{x_n\}$ up to measure zero sets.
This means in specific that $x_n \in \partial G_{\lambda} \cap G_{\lambda}$.
Since $\mm(\{x_n\}) = 2^{-n}$, the minimizer $G_{\lambda}$ contains boundary of positive measure and thus there is no open representative of $G_{\lambda}$.

Notice that the example above has a closed representative since we can always add in the boundary of the set $G_{\lambda}$ to itself since the dirac masses do not add perimeter, and the rest of the boundary is a null set with respect to $\mm$.
\end{example}



\section{Remarks on quasiminimal sets in PI-spaces}\label{sec:PI}

As noted in the Introduction, in PI-spaces we can approximate a domain from inside and outside by uniform domains which are extension domains for $BV$ and Sobolev functions.
Therefore, we will focus here only on connecting our approach of the more general existence result obtained in Section \ref{sec:approximation} with other results on the structure of minimizers in PI-spaces. Here with a PI-space we mean a complete metric measure space $(\X,\sfd,\mm)$ where the measure is doubling and the space satisfies a local $(1,1)$-Poincar\'e inequality. 
Recall that a measure $\mm$ is doubling on $\X$ if there exists a constant $C>0$ so that for every $x \in \X$ and $r>0$ we have
\[
\mm(B(x,2r)) \le C \mm(B(x,r)).
\]
A metric measure space satisfies a local $(1,1)$-Poincar\'e inequality if there exist constants $C > 0$ and $\lambda \ge 1$ so that for every function $f$ in $\X$ with an upper gradient $g_f$, every $x \in \X$ and $r>0$ we have
\[
\int_{B(x,r)} |f-f_{B_r(x)}| \, \d\mm \le Cr \int_{B_{\lambda r}(x)}g_f\,\d\mm,
\]
where $f_A$ denotes the average of $f$ in a set $A \subset X$ of positive and finite measure.
The proof of Theorem \ref{thm:extensionsub} is based on the minimization of the functional
\[
M_\lambda \colon \{A \subset \Omega\} \to [0,+\infty] \colon A \mapsto \text{Per}(A) + \lambda \mm(\Omega \setminus A).
\]
If we replace the term $\lambda\mm(\Omega \setminus A)$
by $\lambda \mm(\Omega \Delta A)$ we obtain a more studied functional
\[
\widetilde M_\lambda \colon \{A \subset \X\} \to [0,+\infty] \colon A \mapsto \text{Per}(A) + \lambda \mm(\Omega \Delta A).
\]
A minimization of the functional $\widetilde M_\lambda$ leads to a set which is close in measure to $\Omega$, but not necessarily contained in $\Omega$. Still, the argument in the proof of Theorem \ref{thm:extensionsub} for showing that the minimizer is a $BV$-extension set works also for the functional $\widetilde M_\lambda$ provided that the minimizer has a closed representative (in order to use Lemma \ref{lma:perimetersum}). Since in general we do not know if the minimizer of $M_\lambda$ or $\widetilde M_\lambda$ has a closed representative, instead of using a global minimizer we took a minimal element in a decreasing chain of closed sets.
Recall that by Example \ref{ex:positiveboundary} we know that the minimizer need not have an open representative.

In PI-spaces we do have a closed representative for the global minimizer of $\widetilde M_\lambda$ in the class of Borel sets. This can be seen via the regularity results of quasiminimal sets. By \cite[Proposition 3.20 and Remark 3.23]{APP2022} we have that in PI-spaces the minimizer of the functional $\widetilde M_\lambda$ is locally $K$-quasiminimal in $\X$. 
Recall that a Borel set $E\subset \X$ is said to be $K$-quasiminimal, or to have $K$-quasiminimal boundary in an open set $\Omega \subset \X$, if for all open $U \Subset \Omega$ and every Borel sets $F,G \Subset U$ we have
\[
{\rm Per}(E,U) \le K {\rm Per}((E \cup F) \setminus G; U).
\]
A set $E$ is said to be locally $K$-quasiminimal in $\Omega$, if instead of requiring the minimality for all open $U \Subset \Omega$ we require that for every $x \in \Omega$ the exists an open neighbourhood $V \subset \Omega$ of $x$ so that for all $U \Subset V$ the above holds. 

By \cite[Theorem 4.2]{KKLS2013} a $K$-quasiminimal set in a PI-space has a representative for which the topological and measure theoretic boundaries agree.
Recall that the measure theoretic boundary of $E$ consists of those points where the (upper) density of both $E$ and $\X\setminus E$ are positive.
By a density point argument, the measure theoretic boundary has always measure zero. Consequently, a $K$-quasiminimal set has both an open and a closed representative. The proof of Theorem \ref{thm:extensionsub} then gives that the closed representative is a $BV$-extension set. However, as we will see in Example \ref{ex:quasicounter}, being a $BV$-extension set is not invariant under taking representatives, so we cannot conclude directly that the open representative is also a $BV$-extension set. 

Notice also that for the functional $M_\lambda$ we have the local $K$-quasiminimality only inside $\Omega$. Therefore, via \cite[Theorem 4.2]{KKLS2013} we only know that the topological boundary of the minimizer has measure zero inside $\Omega$. However, if we start with a domain $\Omega$ with $\mm(\partial\Omega)=0$, we can conclude that also the minimizer of $M_\lambda$ has both an open and a closed representative.

The above argumentation leads to natural questions: In a PI-space, is every domain with locally quasiminimal surface a $BV$-extension set? Is the closure of a domain with locally quasiminimal surface a $BV$-extension set?
We end this section with an example showing that the answer to the first question is negative.
In fact, the example shows that even the open representative of a minimizer of $\widetilde{M}_\lambda$ need not be a $BV$-extension set in a PI-space.
The same example also answers a question in \cite{KKLS2013}: domains with locally quasiminimal surface need not be local John domains in PI-spaces.

 Recall that a domain $\Omega$ is a local John domain if there exist constants $C, \delta > 0$ such that for every $x \in \partial \Omega$, every $0 < r < \delta$ and all $y \in B_r(x)\cap \Omega$ there exists a point $z \in B_{Cr}(x)\cap \Omega$ with $\sfd(y,z) \ge r/C$ and a curve $\gamma \subset \Omega$ such that
 \[
 \ell(\gamma_{y,w}) \le C{\rm dist}(w, \partial \Omega)
 \]
 for all $w \in \gamma$, where $\gamma_{y,w}$ is the shortest subcurve of $\gamma$ joining $y$ and $w$, and $\ell(\alpha)$ denotes the length of a curve $\alpha$.
A motivation for asking about the local John condition comes from the Euclidean setting, where David and Semmes showed that bounded sets with quasiminimal
boundary surfaces are locally John domains \cite{DS1998}.

\begin{example}\label{ex:quasicounter}
Consider the metric measure space $\X = \X_1 \cup \X_2 \cup \X_3$
where $\X_i = \{i\} \times [0,1]$ and for every $t \in [0,1]$ the points $(i,t,0)$, $i= \{1,2,3\}$ are identified. (Later on we will not always write the first coordinate that was above used only as a label.) Let us write the common part of $\X_i$ as $D = \X_1\cap \X_2 \cap \X_3$.
In other words, $\X = [0,1] \times \mathcal T$, with $\mathcal T$ being a tripod with unit length legs.
The distance $\sfd$ on $\X$ is the length distance on each \(\X_i\) given by 
\[
\sfd_{\X_i}(x,y) = |x_1-y_1|+|x_2-y_2|,
\]
and the reference measure $\mm$ is the sum of weighted Lebesgue measures on each $\X_i$: 
\[
\mm = 2\mathcal L^2|_{\X_1} + \mathcal L^2|_{\X_2 \cup \X_3}.
\]
The obtained metric measure space $(\X,\sfd,\mm)$ is an Ahlfors $2$-regular and satisfies the $(1,1)$-Poincar\'e inequality. 
We will consider a domain $\Omega \subset \X$ as $\Omega = \Omega_1 \cup \Omega_2 \cup \Omega_3$, where each $\Omega_i \subset \X_i$ is defined as follows. 
We start by defining as a basic building block a triangle
\[
T = \big\{(x_1,x_2) \in \mathbb R^2\,:\, x_2 > 0,\, x_2 < x_1<1-x_2
\big\}.
\]
Now, for the set \(\Omega_1\) in $\X_1$ we simply choose
\[
\Omega_1 \coloneqq \left([0,1]\times (0,1]\right) \cup J,
\]
the set  $\Omega_2\subset X_2$ is given by
\[
\Omega_2 \coloneqq  \bigcup_{k=0}^\infty \left(2^{-2k-1}T + (2^{-2k-1},0)\right)\cup J,
\]
and the set  $\Omega_3\subset X_3$ is 
given by
\[
\Omega_3 =  \bigcup_{k=0}^\infty \left(\left(2^{-4k-3}T+(2^{-2k-1},0)\right) \cup \left(2^{-4k-3}T + (2^{-2k} -2^{-4k-3},0)\right)\right)\cup J.
\]
The common part $J \subset D$ for the sets above is defined by
\[
J = \bigcup_{k=0}^\infty \left(\left(2^{-2k-1},2^{-2k-1}+2^{-4k-3}\right) \cup \left(2^{-2k} -2^{-4k-3},2^{-2k}\right) \right) \times \left\{0\right\}.
\]
See Figure \ref{fig:nonJohn} for an illustration of the domain $\Omega$.

 \begin{figure}
     \centering
     \includegraphics[width=0.6\textwidth]{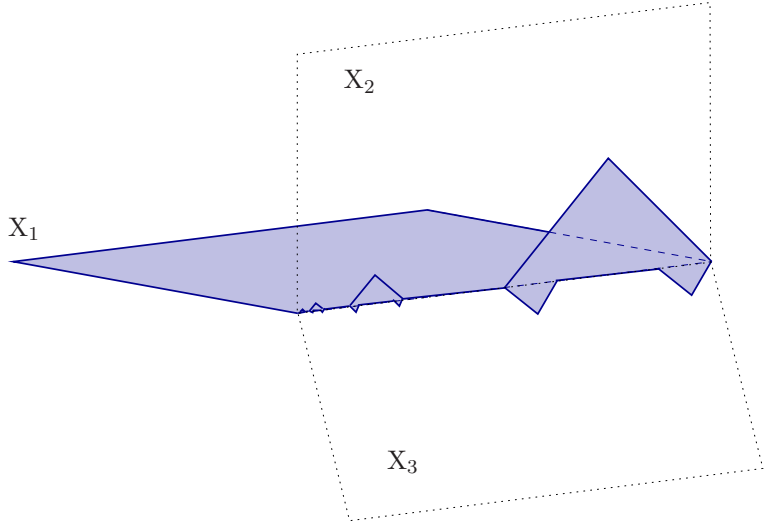}
     \caption{The domain $\Omega$ in Example \ref{ex:quasicounter} lives in three copies of the unit square, $\X_1,\X_2$, and $\X_3$, that are glued together at one edge. The domain minimizes $\widetilde M_\lambda$ and thus has locally quasiminimal surface. One intuitive way to see the quasiminimality is to observe that with local variations one cannot decrease the perimeter much when trying to remove the slits appearing at the common edge in the $\X_1 \cup \X_2$ square. The slits prevent the domain from being locally John or $BV$-extension domain.}
     \label{fig:nonJohn}
 \end{figure}

\noindent
{\color{blue} \sc Claim 1:} \textit{For any \(\lambda\geq 1\),
the domain 
$\Omega$ is a minimizer of $\widetilde M_\lambda$ among Borel subsets of $\X$.}
To prove it, we first show that 
\begin{equation}\label{eq:origperi}
\widetilde M_\lambda(\Omega) = {\rm Per}(\Omega) + \lambda\mm(\Omega \Delta \Omega) = {\rm Per}(\Omega) \le 2.
\end{equation}
This can be verified by simply taking as a sequence \((f_n)_n\)
of Lipschitz functions approaching to $\chi_\Omega$ in $L^1(\mm)$
whose elements \(f_n\)  are given by 
\[
f_n(x) = 1-\min\big(1, n \cdot {\rm dist}(x,\Omega)\big),\quad \text{ for every  }x\in \X.
\]
Then, denoting 
\[
\Omega^n\coloneqq \left\{x \in \X \,:\, 0 < {\rm dist}(x,\Omega) < \frac1n\right\},
\]
we have that 
\[
{\rm Per}(\Omega)\leq \liminf_{n\to +\infty}\int_{\Omega^n}{\rm lip}_a(f_n)\,\sfd \mm\leq \liminf_{n\to +\infty}\, n\cdot \mm(\Omega^n).
\]
Thus, it remains to estimate the measure of \(\Omega^n\).
Notice that by the choice of the distance $\sfd$ and the slopes in the triangle $T$, 
we have for \(k\in \{2,3\}\) that
\[\Omega^n\cap \X_k=\left\{(x,y+a)\in \X_k:\, (x,y)\in \partial \Omega\cap \X_k,\, a\in \Big(0,\frac{1}{n}\Big)\right \}.\]
Therefore, by using Fubini's theorem, 
we get for every \(n\geq 2\) that 
\[
\mm(\Omega^n)
=
\,\mm(\Omega^n\cap \X_2)+\mm(\Omega^n\cap \X_3) = \frac{2}{n}
\]
and accordingly that 
\({\rm Per}(\Omega)\leq 2\).
In order to conclude the proof of the Claim 1, we next show that for any $A \subset \X$ of finite perimeter we have
\begin{equation}\label{eq:anyperi}
 \widetilde M_\lambda(A) \ge 2.
\end{equation}
This follows by showing for any $f \in {\rm LIP}(\X)$
we have
\begin{equation}\label{eq:anyperif}
 \int_\X {\rm lip}_a(f) + \lambda|f-\chi_\Omega| \,\d \mm \ge 2.
\end{equation}
To show this, fix any  $x \in (0,1)$. Then, since \(\lambda\geq 1\), we have that
\[
 \int_{\{(1,x_1)\}\times (0,1)} {\rm lip}_a(f) + \lambda|f-\chi_\Omega| \,\d \mathcal H^1 \ge |f(1,x_1,0)-1|
\]
and for $k\in \{2,3\}$,
\[
 \int_{\{(k,x_1)\}\times (0,1)} {\rm lip}_a(f) + \lambda|f-\chi_\Omega| \,\d \mathcal H^1 \ge |f(k,x_1,0)|.
\]
Combining the above two estimates
and using again a Fubini-type argument taking the choice of our measure into account, we get
\[
\int_\X \lip_a(f)+\lambda |f-\chi_\Omega|\,\d \mm
\geq 2|f(1,x_1,0)-1|+|f(2,x_1,0)|+|f(3,x_1,0)|\geq 2, 
\]
recalling that the points \((i,t,0)\) for \(t\in (0,1)\) and \(i\in \{1,2,3\}\) are identified.
Hence we obtain \eqref{eq:anyperif} and thus \eqref{eq:anyperi}. This proves that $\Omega$ is a minimizer of $\widetilde M_\lambda$.

\noindent
{\color{blue}\sc Claim 2:}
\textit{$\Omega$ is not a $BV$-extension set nor a $\ca{BV}$-extension set.} Towards this, take $k \in \mathbb N$ and define
\[
E_k = 2^{-2k-1}T + (2^{-2k-1},0) \subset \Omega_2.
\]
Then 
\[
{\rm Per}_\Omega(E_k) \le  2^{-4k-2}
\quad \text{ and }\quad \mm(E_k) = 2^{-4k-4}.
\]
However, for any $\tilde E_k \subset \X$ with $\tilde E_k \cap \Omega = E_k$, by looking at the rectangle $\X_1\cup \X_2$, we see that
\[
{\rm Per}_\X(\tilde E_k) \ge 2^{-2k-1}.
\]
Consequently,
\[
\frac{{\rm Per}_\X(\tilde E_k)}{{\rm Per}_\Omega(E_k)} \ge \frac{2^{-2k-1}}{2^{-4k-2}} = 2^{2k+1} \to +\infty, \qquad \text{as }k \to +\infty
\]
and
\[
\frac{{\rm Per}_\X(\tilde E_k) + \mm(\tilde E_k)}{{\rm Per}_\Omega(E_k) + \mm(E_k)} \ge \frac{2^{-2k-1}}{2^{-4k-2} + 2^{-4k-4}} \ge 2^{2k} \to +\infty, \qquad \text{as }k \to +\infty,
\]
proving the claim that $\Omega$ is not a $BV$- nor $\ca{BV}$-extension domain.
(Notice, however, that $\overline{\Omega}$ is a $BV$-extension set.)

\noindent
{\color{blue} \sc Claim 3:}
\textit{$\Omega$ is not locally John domain.} To show  this, we take as the center $x := (0,0) \in \partial \Omega$. Given any $C\ge 1$ and $\delta>0$ we take $k \in \mathbb N$ large enough so that 
\[
 r_k := \sqrt{2}\cdot 2^{-4k-1} < \delta/C \qquad \text{and} \qquad
2^{2k+1} > C.
\]
Now, take $r = Cr_k$ and select $y = (2^{-2k-1}+2^{-2k-2},2^{-2k-3}) \in E_k \subset B_r(x)$. Notice that by the selection of $k$ we have $0 < r < \delta$.
Then
\[
E_k \subset B_{r_k}(y) = B_{r/C}(y),
\]
so the point $z$ in the John condition is forced to be selected outside $E_k$. Consequently, any curve $\gamma$ joining $y$ and $z$ in $\Omega$ must pass through a point
\[
w \in \left(\left(2^{-2k-1},2^{-2k-1}+2^{-4k-3}\right) \cup \left(2^{-2k} -2^{-4k-3},2^{-2k}\right) \right) \times \left\{0\right\} \subset J.
\]
We then have
\[
\frac{{\rm dist}(w,\partial\Omega)}{\ell(\gamma_{y,w})} \le
\frac{2^{-4k-4}}{2^{-2k-3}} = 2^{-2k-1} < \frac{1}{C},
\]
where in the last inequality we used again the selection of $k$. This contradicts the John condition with the given parameters $C$ and $\delta$.
\end{example}

\begin{remark}
Notice that as a minimizer of $\widetilde M_\lambda$ in a PI-space, the domain $\Omega$ of Example \ref{ex:quasicounter} also has quasiminimal surface. If we use as the measure $\mm$ in the example the 2-dimensional Hausdorff measure, we have that the space $(\X,\sfd,\mm)$ is isotropic. (Let us recall that a metric measure space is isotropic whenever the density function \(\theta_E\) associated with the set of finite perimeter \(E\) 
and for which it holds that \({\rm Per}(E,\cdot)=\theta_E\mathcal H\restr{\partial^eE}\)
is independent on the set \(E\) itself. We refer to \cite{Ambrosio2002} for more details about the mentioned density function.)
Since the property of being quasiminimal is invariant under a change of the reference measure to a comparable one, we will thus obtain a version of the example where the space is isotropic, but the domain only has quasiminimal surface instead of being a minimizer of $\widetilde M_\lambda$. Notice also, that changing to a distance $\sfd$ induced by the Euclidean distances in $\X_i$ we also preserve the quasiminimality, since the change in distance is bi-Lipschitz.
\end{remark}

\section{Open questions}\label{sec:q}

Our extension result leads to several questions that we have not yet been able to answer. In Theorem \ref{thm:extensionsub} we proved that we can approximate domains from inside by closed $BV$-extension sets. For the special case of PI-spaces, 
in Section \ref{sec:PI}, we noted that minimizers of $\widetilde M_\lambda$ have also open representatives.
However, Example \ref{ex:quasicounter} showed that the open representatives need not be $BV$-extension sets even in PI-spaces. What still remained open is if being a minimizer of $\widetilde M_\lambda$ is really needed or if having just quasiminimal surface is enough:

\begin{question}
Let $(\X,\sfd,\mm)$ be a PI-space and $\Omega \subset \X$ a bounded domain with locally $K$-quasiminimal surface. Is then
$\overline{\Omega}$ a $BV$-extension set?
\end{question}



Another question stemming from the proof of Theorem \ref{thm:extensionsub} is if we really need to take the partial order into use to guarantee that the minimal element has a closed representative.

\begin{question}\label{q:closedrep}
Let $(\X,\sfd,\mm)$ be a metirc measure space and $\Omega \subset \X$ a bounded domain.
Let $E$ be a minimizer of $M_\lambda$ (or $\widetilde M_\lambda$) among Borel subsets of $\Omega$ (or $\X$ respectively). Does $E$ have a closed representative?
\end{question}

For PI-spaces the answer to Question \ref{q:closedrep} is positive for $\widetilde M_\lambda$, see again Section \ref{sec:PI}.

Independent of the minimization approach, the obvious question still remaining is:

\begin{question}
 Let $(\X,\sfd,\mm)$ be a metric measure space, $\Omega \subset \X$ a bounded domain and $\varepsilon > 0$.
 Does there exist a $BV$-extension domain $A \subset \Omega$ such that $\mm(\Omega\setminus A)< \varepsilon$?
\end{question}

None of our approximations is from outside because we argue that the minimizer is an extension set by comparing the value of the functional to value at a modification of the minimizer where we take away an open subset. 

\begin{question}
 Let $(\X,\sfd,\mm)$ be a metric measure space, $\Omega \subset \X$ a bounded domain and $\varepsilon > 0$.
 Does there exist a $BV$-extension domain (or just a $BV$-extension set) $A \supset \Omega$ such that $\mm(A \setminus \Omega)< \varepsilon$?
\end{question}

In addition to knowing the answer to the above questions, it would be interesting to see if we can also approximate domains by Sobolev $W^{1,p}$-extension domains in the absence of the local Poincar\'e inequality. In particular, the case $p=1$ is intimately connected to the $BV$ and perimeter extensions even in general metric measure spaces \cite{CKR2023}.

\end{document}